\documentclass[final,3p,times]{elsarticle}
\usepackage{amsmath,amssymb,amsthm}
\usepackage[all]{xy}
\usepackage{graphicx}
\usepackage{color}
\usepackage{circledsteps}

\usepackage[mathscr]{euscript}

\newcommand{\Msum}{{\, \Circled{\Sigma}\, }}

\newcommand{\C}{{\mathbb C}}

\newcommand{\N}{{\mathbb N}}

\newcommand{\R}{{\mathbb R}}

\newcommand{\Z}{{\mathbb Z}}

\theoremstyle{plain}
\newtheorem{Theorem}{Theorem}[section]

\newtheorem{Corollary}[Theorem]{Corollary}

\newtheorem{Lemma}[Theorem]{Lemma}
\newtheorem{Proposition}[Theorem]{Proposition}

\theoremstyle{definition}
\newtheorem{Definition}[Theorem]{Definition}
\newtheorem{Remark}[Theorem]{Remark}
\newtheorem{Example}[Theorem]{Example}

\newcommand{\cA}{{\mathscr A}}
\newcommand{\tV}{{\tilde V}}

\newcommand{\cSigma}{{\mathscr S}}

\newcommand{\supp}{{\emph Esupp\,}}
\newcommand{\suppf}{{\emph supp\,}}

\newcommand{\costi}[1]{{\color{black}#1}}

\title{Multivariate compactly supported $C^\infty$  functions by subdivision}

\author[1]{Maria Charina}
\author[2]{Costanza Conti}
\author[3]{Nira Dyn}

\address[1]{Fakult\"at f\"ur Mathematik, Universit\"at Wien, Austria}
\address[2]{Dipartimento di Ingegneria Industriale, University of Florence, Italy}
\address[3]{School of Mathematical Sciences, Tel Aviv University, Israel}

\begin{document}

\begin{abstract}
This paper discusses the generation of multivariate $C^\infty$ functions with  
compact small supports by subdivision schemes. Following the construction of such a univariate function, called \emph{Up-function}, by a non-stationary scheme based on masks of {spline subdivision schemes}
of growing degrees, we term the multivariate functions we generate Up-like functions. We generate them by non-stationary schemes based on
masks of  box-splines of growing supports.
To analyze the convergence and smoothness of these non-stationary schemes, we develop new tools 
 which apply to a wider class of schemes than the class we study. With our method for achieving small compact supports, we obtain, in the univariate case, Up-like functions with supports $[0, 1 +\epsilon ]$ in comparison to the support $[0, 2] $ of the
Up-function. Examples  of univariate and bivariate Up-like functions are given. As in the univariate case, the construction of Up-like functions can motivate the generation  of $C^\infty$ compactly supported wavelets of small support in any dimension.

\end{abstract}


\maketitle

\emph{\bf We dedicate this paper to the memory of  Maria Charina, who was the initiator of this work and took a very active part in it. Maria was a great collaborator, hard-working, sensitive, always smiling and with positive attitude to her work and her co-workers.
More importantly, she was a very good friend.}

\section{Introduction}
In this paper we design multivariate non-stationary subdivision schemes which generate $C^\infty$ compactly supported basic limit functions with small supports. We term these functions  \emph{Up-like functions}, following the construction of the univariate Up-function studied in \cite{DynLevinDerfel}. While stationary subdivision schemes use the same refinement rules in each refinement level, non-stationary subdivision schemes apply level dependent refinement rules. The variety of refinement rules is
what makes this schemes capable of generating limits which cannot be generated by stationary subdivision schemes, such as
 exponential splines or conic sections. In these non-stationary examples, the generating schemes have level dependent 
refinement rules but  the supports of all the different rules is the same, a fact which is taken into account when analyzing their convergence and smoothness (see e.g. \cite{DynLevin_asymp} or the survey paper \cite{ContiDyn2019}). Differently, the non-stationary subdivision schemes we study in this paper, are based on refinement rules of growing support size. While the smoothness of the functions generated by stationary schemes is bounded by the support size of their basic limit functions, here we show how non-stationary schemes can be used to generate multivariate $C^ \infty$ basic limit functions of compact and small supports. 

In \cite{DynLevinDerfel} (see also \cite[Example 4.18]{DynLevin_acta}), it is shown that the basic limit function of a univariate non-stationary scheme which applies at refinement level $r$ the refinement rules of the stationary scheme generating polynomial  splines of degree $r$ for $r\in \N_0$, is a $C^\infty$ function supported on $[0,2]$. Using a similar approach, with multivariate non-stationary schemes based on refinement rules of  three-directional box-splines, we show how to generate multivariate compactly supported $C^\infty$ basic limit functions.  In fact, our method applies to other types of box-splines schemes as well as to other types of schemes with positive masks. 
We also present a method for obtaining $C^\infty$ functions of small supports. With this method, in the univariate case, we design a $C^\infty$ functions supported in $ [0, 1+\epsilon]$  with positive $\epsilon$ arbitrarily small.                      
We term all the proposed variants of the Up-function, \emph{Up-like functions} and prove the convergence and smoothness of their generating non-stationary schemes. We also analyze the supports of the Up-like functions, and give univariate and bivariate examples. 

For deriving the above results, we develop new analysis tools for non-stationary schemes. In particular, we prove a general convergence theorem for non-stationary schemes whose refinement rules are in a certain class. For the smoothness analysis of our non-stationary schemes we use multivariate smoothing factors. The convergence and smoothness theorems we provide apply to a wider class of non-stationary schemes than the class of schemes generating our Up-like functions. 

Note that the idea of generating $C^\infty$  univariate functions of compact support is used  to generate univariate  $C^\infty$ wavelets of compact supports (see e.g. \cite{CohenDyn, HanShen}).
The construction of these wavelets is by a non-stationary subdivision scheme which applies in refinement level $k$ refinement rules 
of the $k$-th order filter of the Daubechies wavelets \cite{Dau}.
The Up-like functions we construct in this paper can also motivate the generation of  multivariate $C^\infty$ compactly supported wavelets.
We also note that, perhaps, the construction and analysis of $C^\infty$ functions with
compact small supports proposed here, could be performed using Fourier techniques as is done in the univariate case in \cite[Secion 2]{HanShen}.  Even if it might be less general, we believe that our time domain approach  is elegant and easy to follow.

\smallskip The rest of the paper is organized as follows. After two subsections dealing with background, notation and summary of our proposed new approach, Section \ref{sec:convergence_up} presents convergence results of non-stationary subdivision schemes with masks in a special class. Then, Section \ref{sec:regularity_up} is devoted to the smoothness analysis of subdivision schemes with a certain number of multivariate smoothing factors in their symbols. In Section \ref{sec:compact support}, after stating an assumption on the supports of the masks of the non-stationary schemes  which ensures the compact
support of their basic limit functions, we compute the precise size of these supports. The closing Section \ref{sec:examples} discusses several examples  of Up-like functions generated by special 
non-stationary schemes based on B-splines' masks and on three directional box-splines' masks. 

\subsection{Background and notation} \label{BACK}
A \emph{subdivision operator} $S_a: \ell_\infty(\Z^d) \rightarrow \ell_\infty(\Z^d)$ is defined by
\begin{equation}\label{def:sub}
 S_{a} f=\sum_{\beta\in \Z^d} \, a(\cdot-2\beta) \, f(\beta), \qquad f \in \ell_\infty(\Z^d),
\end{equation}
with the real-valued sequence $a=\left(a(\alpha)\, : \, \alpha\in \Z^d\right)$ called the subdivision \emph{mask}. 
The  \emph{support} of the mask is the set $\{ \alpha\in \Z^d\,|\ a_\alpha\neq 0\}$, here assumed to be finite. In our analysis 
we also need the notion of \emph{extended support} of a mask $a$,
 $\supp(a)$, defined as 
\begin{Definition} \label{def:support}
$$
\supp(a)=\overline{Convexhull\left(\{ \alpha\in \Z^d\,|\ a_\alpha\neq 0\}\right)}.
$$
\end{Definition} 
\noindent Note that, by definition, $\supp(a)$ is a compact and convex set.

A mask $a$ consists of $2^d$ sub-masks $\{a(\epsilon+2\alpha),\ \alpha\in \Z^d\},\ \epsilon \in \{0,1\}^d$. Each sub-mask corresponding to one 
$\epsilon \in \{0,1\}^d$
generates in \eqref{def:sub} 
 the refinement rule
$$
(S_{a} f)(\epsilon+2\alpha)=\sum_{\beta\in \Z^d} \, a(\epsilon+2\beta) \, f(\alpha-\beta), \quad   \alpha\in \Z^d, \qquad f \in \ell_\infty(\Z^d).
$$

\smallskip A \emph{stationary subdivision scheme} is based on the repeated application of the subdivision operator $S_a$,  
and generates the sequence of refined values 
\begin{equation} \label{def:stationary_subdivision}
 f^{[r]}=S^r_{a}  f^{[0]},  \quad f^{[0]} \in \ell_\infty(\Z^d),  \quad r \in \N.
\end{equation} 
In contrast, a \emph{non-stationary subdivision scheme} is based on the repeated application of \emph{different} subdivision operators at different refinement levels, $(S_{a_k} \, :\, k \in \N_0)$, defined by
\begin{equation} \label{def:non_stationary_subdivision}
 f^{[r]}=S_{a_{r-1}} \ldots S_{a_0} f^{[0]},  \quad f^{[0]} \in \ell_\infty(\Z^d),  \quad r \in \N.
\end{equation}

The definitions of \emph{convergence} of stationary and non-stationary schemes 
vary insignificantly only by the choice of the 
subdivision operators in the subdivision recursion. Thus, we define the concept of
convergence only in the more general, non-stationary case. It uses the notion of {\it infinity norm} of a sequence $f\in  \ell_\infty(\Z^d)$, defined as 
$ \| \,  f \, \|_\infty=sup_{\alpha\in\Z^d}|f(\alpha)|$. In this paper we use only the $\|\cdot\|_\infty$, therefore we omit the subscript $\infty$.

\begin{Definition} \label{def:convergence}
Let $s\in \N_0$. A subdivision scheme $(S_{a_k}\, : \,   k\in \N_0)$ \emph{converges to $C^s(\R^d)$ limit functions}, if
for any initial sequence $f^{[0]} \in \ell_\infty(\Z^d)$, there exists a
\emph{function} $g \in C^s(\R^d)$, nonzero for at least one nonzero sequence $f^{[0]}$, such that
\begin{equation}\label{eq:subdivisionlimit}
  \lim_{k \to\infty}  \| \, g(2^{-(k+1)}\cdot)- S_{a_k} S_{a_{k-1}} \ldots S_{a_0} f^{[0]} \, \|= 0.
\end{equation}
\end{Definition}
We write \eqref{eq:subdivisionlimit} as
$$
 g=\lim_{r \rightarrow \infty} S_{a_r} S_{a_{r-1}} \ldots S_{a_0} f^{[0]}, \qquad f^{[0]} \in \ell_\infty(\Z^d).
$$

The convergence of a stationary scheme based on $S_a$ implies that (see e.g. \cite{CDM91})
\begin{equation}\label{5}
\sum_{\alpha\in \Z^d}a(\alpha+\epsilon)=1,\quad \hbox{for}\quad \epsilon \in \{0,1\}^d.
\end{equation}
We continue with some \emph{factorization} properties for which we make use of the masks' \emph{symbol}
$$
 a(z)=\sum_{\alpha \in \Z^d} \, a(\alpha) \, z^\alpha, \qquad z^\alpha=z_1^{\alpha_1} \cdot \ldots \cdot z_d^{\alpha_d}, \qquad
 z \in \left( \C \setminus \{0\}\right)^d,
$$
and of the notion of \emph{norm of $S_a$} defined by (see e.g. \cite{CDM91})
\begin{equation} \label{eq:norma}
 \|S_{a}\| = \max_{\epsilon \in \{0,1\}^d} \sum_{\alpha \in \Z^d} \,| \, a(\alpha+\epsilon)\, | \,.
\end{equation}
For a vector $v \in \N_0^d$, $v \not=0$, we define the \emph{directional difference operator}
$$
 \nabla_v: \ell_\infty(\Z^d) \rightarrow \ell_\infty(\Z^d), \qquad \nabla_v \, f=f(\cdot +v)-f,\quad f\in  \ell_\infty(\Z^d),
$$ 
and, for a basis $V=\{v^{[1]}, \ldots v^{[d]}\}$ of $\R^d$ with $v^{[1]}, \ldots v^{[d]} \in \N_0^d$,  we define the difference operator as the vector
$$
  \nabla_V: \ell_\infty(\Z^d) \rightarrow \ell_\infty^d(\Z^d), \qquad  \nabla_V \, f =\left(\begin{array}{cc} \nabla_{v^{[1]}} \, f \\
 \vdots \\
 \nabla_{v^{[d]}}  \, f 
 \end{array}\right)\in \R^d, \quad f\in  \ell_\infty(\Z^d).
$$ 
 In the stationary case, corresponding to the scheme $(S_{a}\, : \,   k\in \N_0)$ which is  simply denoted by $S_{a}$, it is well known \cite{CDM91, Dyn} that the convergence of $S_a$ implies the existence of the decomposition
\begin{equation} \label{eq:SaSb}
 \nabla_V \, S_a = S_b \,  \nabla_V,
\end{equation}
with $V$ the canonical basis in $\R^d$ and with a matrix-valued difference subdivision scheme $$S_b:\ell_\infty^d(\Z^d) \rightarrow \ell_\infty^d(\Z^d),$$
which is \emph{contractive}, namely there exists $\ell\in \N$ such that $\|S_b^\ell\|<1$. 

\begin{Remark} \label{def:BLF} 
For the sake of simplicity, when $V$ is the canonical basis in $\R^d$, the difference operator $ \nabla_V$ is denoted by $\nabla$ in the rest of the paper
\end{Remark}

\begin{Definition} \label{def:BLF} 
The \emph{basic limit function} of a converging scheme $S_a$ is the compactly supported function
\begin{equation} \label{def:phia}
 \phi_a=\lim_{r \rightarrow \infty} S^r_a \delta, \quad \hbox{with}\quad \delta_\alpha\neq 0\ \hbox{only for}\ \alpha=0, \ \hbox{and}\  \delta_0=1.
\end{equation}
\end{Definition}

\begin{Definition} \label{def:LcontractivityNS} 
A convergent subdivision operator $S_a$ is called \emph{$(L,\rho)$-level contractive}, if
$$L=argmin\, \left\{n\in\mathbb N\ :\ \|S_{b}^n\|<1\right\}\quad \hbox{and} \quad
\rho=\|S_{b}^{L}\| <1.$$ 
\end{Definition}

We conclude the background subsection with the notion of \emph{Minkowski sum} of two subsets $H,\, G$ of $\R^d$ and the \emph{product  by a scalar $\lambda \in \R$} of such a set defined,
respectively, as 
\begin{equation*}
 H\Msum G=\{h + g\,:\, h \in H, \, g \in G\} \quad  \hbox{and} \quad \lambda\,  H=\{\lambda \cdot h\,:\, h \in H\}\,. 
\end{equation*}

\subsection{The proposed new approach} \label{sec:newapproach}
In this paper, we analyze the convergence to the basic limit function
\begin{equation} \label{def:phim}
 \phi_0=\lim_{k \rightarrow \infty} S_{a_k} S_{a_{k-1}} \ldots S_{a_0} \delta,  
\end{equation}
generated by the non-stationary subdivision scheme $(S_{a_k} \, :\, k \in \N_0)$ with all masks $a_k$ in a special class, using a novel approach. The novel approach links the convergence 
to $\phi_0$ with 
the convergence  of the sequence of limits 
\begin{equation} \label{def_phi_k}
 \varphi_k=\lim_{r \rightarrow \infty} S_{a_k}^r S_{a_{k-1}} \ldots S_{a_0} \delta,\quad k\in \N_0,
\end{equation}
of the corresponding stationary schemes. 

The smoothness analysis of $ \phi_0$ is based on  special classes of masks whose symbols contain a certain number of multivariate smoothing factors. 

Concerning the support of $\phi_0$, the novelty is that instead of considering operators based on a sequence of masks with the same support as done in the majority of works on non-stationary schemes, we involve masks of slowly growing supports. This allow us to  achieve $C^\infty$ basic limit functions with compact, small, supports.

\section{Convergence analysis} \label{sec:convergence_up}

\noindent In this section, we show convergence of non-stationary subdivision schemes with masks in a special class, denoted by ${\mathscr A}(L,\rho)$. 
\begin{Definition} \label{def_A0N}
Let $L\in \N, \ L\ge 1$ and $\rho\in \R,\ \rho<1$. We call ${\mathscr A}(L,\rho)$   the class of masks $a$ with the following three properties:
\begin{description}
\item[$1)$] The mask entries are non-negative;
\item[$2)$] The stationary subdivision scheme $(S_{a}\, : \,   k\in \N_0)$  is convergent;
\item[$3)$] For each mask $a\in {\mathscr A}(L,\rho)$, the subdivision operator $S_a$ is $(L(a),\rho(a))$-level contractive with $L(a)\le L$ and
$\rho(a)\le \rho$.
      \end{description}
      \end{Definition}
For this class of masks we prove two preliminary results used in the proof of the main convergence result of this section,  
Theorem~\ref{th:convergence}.

\begin{Lemma} \label{lemma:new} Assume $a_k \in {\mathscr A}(1,\rho),\, k\in \N_0$, with $\rho<1$.  Then
\begin{eqnarray*}
 \lim_{k,r \rightarrow \infty} \|  S_{a_k}^{r+1} S_{a_{k-1}} \ldots S_{a_0} \delta - S_{a_{k+r}} S_{a_{k+r-1}} \ldots S_{a_0} \delta\|  =0\,.
\end{eqnarray*}
\end{Lemma}
\begin{proof}
For $r\in \N$ we use the relation 
$$
 S_{a_{k}}^{r+1}   S_{a_{k-1}} \ldots S_{a_0} - S_{a_{k+r}} S_{a_{k+r-1}} \ldots S_{a_0} =   \sum_{j=1}^r S_{a_{k+r}} \ldots S_{a_{k+j+1}} ( S_{a_{k}}-S_{a_{k+j}})S^j_{a_{k}} S_{a_{k-1}}\ldots S_{a_{0}}
$$

with the convention that the product of the operators $S_{a_{k+r}} \ldots S_{a_{k+j+1}}$ is the identity for $j+1> r$. This equality is easy to see, since the sum on the right hand-side is a telescoping sum. Therefore,
\begin{equation}\label{8p}
\|S_{a_{k}}^{r+1}    S_{a_{k-1}} \ldots S_{a_0} \delta- S_{a_{k+r}} S_{a_{k+r-1}} \ldots S_{a_0}\delta\|\le    
\sum_{j=1}^r \| S_{a_{k+r}} \ldots S_{a_{k+j+1}}( S_{a_{k}}-S_{a_{k+j}})S^j_{a_{k}} S_{a_{k-1}}\ldots S_{a_{0}}\delta\|.
\end{equation}
The convergence of $ S_{a_{k}}$ and $ S_{a_{k+j}}$ implies, by \cite[Lemma 3]{CharinaContiSauer}, the existence of the  
operator 
$$
  S_{D_{k, j}}: \ell_\infty^{d}(\Z^d) \rightarrow \ell_\infty(\Z^d), \qquad S_{a_{k+j}}-S_{a_{k}} = S_{D_{k,j}} \nabla.
$$
Obviously, $\|S_{D_{k,j}}\nabla \|\le 2$ since for any $k\in \N_0$, $\|S_{a_k}\|=1$,  due to \eqref{5} and the non-negativity of the masks' entries.
Using \eqref{eq:SaSb} repeatedly we get from \eqref{8p},
$$
\begin{array}{lll}
\|S_{a_{k}}^{r+1}  S_{a_{k-1}} \ldots S_{a_0} \delta- S_{a_{k+r}} S_{a_{k+r-1}} \ldots S_{a_0}\delta \|&\le &  
\displaystyle{\sum_{j=1}^r \| S_{a_{k+r}} \ldots S_{a_{k+j+1}}\|} \cdot \|( S_{D_{k,j}}\, \nabla)S^j_{a_{k}} S_{a_{k-1}}\ldots S_{a_{0}}\delta\|\\
&\le& \displaystyle{\sum_{j=1}^r \| S_{a_{k+r}} \ldots S_{a_{k+j+1}}\|} \cdot \|S_{D_{k,j}}\,|_\nabla\|\cdot \|S^j_{b_{k}}\,|_\nabla\|\cdot \| S_{b_{k-1}}\ldots S_{b_{0}}\nabla \delta\|\\
&\le& 2^{d+1}\| S_{b_{k-1}}\ldots S_{b_{0}}\nabla \delta\| \cdot \displaystyle{ \sum_{j=1}^r \| S_{a_{k+r}} \ldots S_{a_{k+j+1}}\|\cdot \|S^j_{b_{k}}\,|_\nabla\| },
\end{array}
$$
where we used the fact that $\|\nabla \delta\|\le 2^d$ and that $ \|S_{D_{k,j}}\,|_\nabla\| = \|S_{D_{k,j}}\nabla \|\le 2$.  Next, due to $a_k \in {\mathscr A}(1,\rho), \, k\in \N_0$, and we conclude in view of \eqref{5}, \eqref{eq:norma}, \eqref{eq:SaSb} that $\|S_{a_k}\|=1$, $\|S_{b_k}\|\le \rho$ and therefore
$$
 \|S_{b_k} |_{\nabla}\| \le  \|S_{b_k} \| \le \rho < 1, \quad k \in \N_0,\quad \hbox{and}\quad \| S_{a_{k+r}} \ldots S_{a_{k+i+1}}\| \le 1 .
$$
Thus,
\begin{eqnarray*}
\|S_{a_{k}}^{r+1}  S_{a_{k-1}} \ldots S_{a_0} \delta- S_{a_{k+r}} S_{a_{k+r-1}} \ldots S_{a_0}\delta \| &\le&  
 2^{d+1} \, \rho^{k} \,  \sum_{j=1}^{r}  \,\rho^j\, 
\end{eqnarray*}
Finally, letting $r$ and $k$ go to infinity, we obtain
$$
  \lim_{k,r \rightarrow \infty} \|  S_{a_k}^r S_{a_{k-1}} \ldots S_{a_0} \delta - S_{a_{k+r}} S_{a_{k+r-1}} \ldots S_{a_0} \delta\|  \le \frac{2^{d+1}}{1-\rho} \cdot \lim_{k \rightarrow \infty} \, \rho^k  = 0\,,
$$
 which completes the proof of the Lemma.\end{proof}

The next Proposition~\ref{prop:main_up} studies the properties of the
function-sequence $\{ \varphi_k \}_{k\in \N_0}$, where the function $\varphi_k$ is the limit of the stationary scheme $S_{a_k}$ applied
to the initial sequence 
$S_{a_{k-1}} \ldots S_{a_0} \delta\in \ell_\infty(\Z^d)$ as in \eqref{def_phi_k}. The functions $\varphi_k,\ k\in\N_0$ are continuous as limits of stationary schemes.

\begin{Proposition} \label{prop:main_up} 
Assume 
$a_k \in {\mathscr A}(1,\rho)$,  $k\in \N_0$, with $\rho<1$. Then, the sequence of continuous functions
\begin{equation}\label{th:convergence_aux1}
\{ \varphi_k \}_{k\in \N_0},\quad \hbox{with}\quad  \varphi_k=\lim_{r \rightarrow \infty} S_{a_k}^r S_{a_{k-1}} \ldots S_{a_0} \delta, \quad k \in \N_0,
\end{equation}
is a Cauchy sequence, with a nonzero, continuous limit. 
\end{Proposition}

\begin{proof} Due to the convergence of each stationary subdivision scheme $S_{a_k}$, $k \in N_0$, the limits $\varphi_k$ in \eqref{th:convergence_aux1} exist  and are at least continuous 
for $k\in \N_0$. For $m,n,r \in \N$, we write
{\begin{eqnarray*}
 \|\varphi_{m+n}-\varphi_m\|  \le \|\varphi_{m+n}-S_{a_{m+n}}^{r} S_{a_{m+n-1}} \ldots S_{a_0} \delta\| &+&
 \|\varphi_{m}-S_{a_{m}}^{r+n} S_{a_{m-1}} \ldots S_{a_0} \delta\|  \\ 
 &+& \|S_{a_{m+n}}^r S_{a_{m+n-1}} \ldots S_{a_0} \delta - S_{a_{m}}^{r+n} S_{a_{m-1}} \ldots S_{a_0} \delta \| \\
 \\
\le  \|\varphi_{m+n}-S_{a_{m+n}}^{r} S_{a_{m+n-1}} \ldots S_{a_0} \delta\| &+&
 \|\varphi_{m}-S_{a_{m}}^{r+n} S_{a_{m-1}} \ldots S_{a_0} \delta\|  \\ 
\|S_{a_{m+n}}^r S_{a_{m+n-1}} \ldots S_{a_0} \delta - S_{a_{m+n+r-1}} \ldots S_{a_0}  \delta \| &+&  \|S_{a_{m+n+r-1}} \ldots S_{a_0} \delta -S_{a_{m}}^{r+n} S_{a_{m-1}} \ldots S_{a_0} \delta \|
\end{eqnarray*}}
Thus, due to the convergence of each $S_{a_k}$ the first two terms in the last sum above converge to zero
when $r$ tends to infinity while by Lemma \ref{lemma:new}, the last two terms converge to zero when $n,r$ tend to infinity.
Therefore,  $\{\varphi_k\}_{k\in\N_0}$ is a Cauchy sequence of continuous functions uniformly converging to a continuous limit, $\varphi$. This limit is nonzero due to  \eqref{5} and the assumption $a_k \in {\mathscr A}(1,\rho),\ k\in \N_0$.
\end{proof}

Now we are ready to prove our main convergence result.

\begin{Theorem} \label{th:convergence} A non-stationary scheme $(S_{a_k} \, : \,  k\in \N_0)$ with $a_k \in {\mathscr A}(1,\rho),  \, k\in \N_0$ , and $\rho<1$, is convergent. 
\end{Theorem}

\begin{proof}  By Proposition~\ref{prop:main_up}, there exists a nonzero limit 
of the sequence in \eqref{th:convergence_aux1}. We show that $\varphi$ is the basic limit
function of the non-stationary scheme, i.e. that
\begin{equation}\label{aux}
  \lim_{k \to\infty}  \| \varphi(2^{-(k+1)}\cdot)- S_{a_k} S_{a_{k-1}} \ldots S_{a_0} \delta \|= 0.
\end{equation}
The convergence in \eqref{aux} is ensured, if for any $\epsilon>0$ and for large enough $r \in \N$  and $k \in \N_0$ 
\begin{eqnarray*}\
  \| \varphi(2^{-{(k+r+1)}}\cdot)- S_{a_{k+r}} S_{a_{k+r-1}} \ldots S_{a_0} \delta \| &\le&
	\| \varphi(2^{-{(k+r+1)}}\cdot)-  \varphi_k(2^{-{(k+r+1)}}\cdot) \| + 
	\| \varphi_k(2^{-{(k+r+1)}}\cdot)-  S_{a_k}^{r+1} S_{a_{k-1}} \ldots S_{a_0} \delta \| \\
	&+& \|  S_{a_k}^{r+1} S_{a_{k-1}} \ldots S_{a_0} \delta - S_{a_{k+r}} S_{a_{k+r-1}} \ldots S_{a_0} \delta\| < \epsilon\,.
\end{eqnarray*} 
Due to the convergence of each stationary scheme $S_{a_k}$ we know that the second summand above goes to zero as $r\rightarrow \infty$, and by Proposition~\ref{prop:main_up} we know that the first summand goes to zero as $k\rightarrow \infty$.  Since due to 
Lemma \ref{lemma:new} we have that
$$
 \lim_{k,r \rightarrow \infty} \|  S_{a_k}^{r+1} S_{a_{k-1}} \ldots S_{a_0} \delta - S_{a_{k+r}} S_{a_{k+r-1}} \ldots S_{a_0} \delta\|  =0\, ,
$$
the proof is completed.
\end{proof}

\begin{Remark}\label{important}
Note that from the above theorem we have that $\phi_0\equiv \varphi=\lim_{k\rightarrow \infty}\varphi_k$.
\end{Remark}

A direct consequence of the method of proof of Theorem \ref{th:convergence}, leads to the following more general result \costi{that involves masks in $\in {\mathscr A}(L,\rho)$ with $L>1$ and $\rho<1$}.
\begin{Theorem} \label{teo:convergence2} A non-stationary scheme $(S^{L(a_k)}_{a_k} \, : \,  k\in \N_0)$  is convergent if $\sup_{k\in \N_0}L(a_k) \le L<\infty$, $\sup_{k\in \N_0}{\rho(a_k)}\le \rho <1$, and  $a_k \in {\mathscr A}(L,\rho),  \, k\in \N_0$. 
\end{Theorem}


\begin{Remark} \label{rem:convergence} 
In the same fashion, we can define a convergent non-stationary scheme $(S^{m_k \cdot L(a_k)}_{a_k} \, : \,  k\in \N_0)$  
with the operator $S_{a_k}$ repeated $m_k \cdot L(a_k)$ times for arbitrary $m_k \in \N$. 
\end{Remark}

\section{Smoothness analysis} \label{sec:regularity_up}

This section is devoted to the smoothness analysis of subdivision schemes with masks in special classes characterized by symbols containing different numbers of multivariate smoothing factors.  In particular, Corollary~\ref{coro:3} allows us to conclude the smoothess of our multivariate Up-like functions, which is one of the main goals of our paper.

\subsection{Smoothness of stationary schemes via smoothing factors} 
We start with the definition of a multivariate smoothing factor and of its symbol.

\begin{Definition}
A {\it symbol of a directional smoothing factor}  in ${\mathbb R}^d$ in direction $v\in {\mathbb N_0^d}$ is  $$\frac{1+z^ v}{2},\quad \hbox{where}
 \quad z=(z_1,\ldots,z_d),\ \  v=(v_1,\ldots, v_d),\ \ z^v=(z_1^{v_1},\ldots,z_d^{v_d}),\quad  z \in \left( \C \setminus \{0\}\right)^d.$$
The {\it symbol of a (full) smoothing factor} in ${\mathbb R}^d$, $s_V(z)$,  is a product of the symbols of $d$ directional smoothing factors in $d$ linearly independent directions in ${\mathbb N_0^d}$
\begin {equation}
\label {fullsf}
 s_V(z)= {\displaystyle \Pi_{j=1}^d\frac{1+z^{ v^{[j]}}}{2} },\quad V=\{v^{[1]},\ldots,v^{[d]}\},\, \quad  z \in \left( \C \setminus \{0\}\right)^d.
\end{equation} 
\end{Definition}

We continue by presenting four results, Propositions \ref{prop:result1}, \ref{prop:result2}, \ref{prop:36} and \ref{prop:result3},  on stationary schemes which are relevant to the convergence and smoothness analysis of non-stationary schemes with masks of growing supports.

\begin{Proposition}\label{prop:result1}
 Let  $S_a$ be a converging stationary subdivision scheme with $a(z)$  a symbol  with positive coefficients.  Let $V=\{v^{[1]},\ldots,v^{[d]}\}$ be $d$ linear independent directions of ${\mathbb N_0^d}$,  and  let 
$ D\subset  {\mathbb N_0^d}$ be a collection of directions with possible repeated directions that are considered as different elements of $D$.  The  symbol 
	\begin{equation}\label{def:c}c(z)=a(z)\left(\Pi_{v\in D}\frac{1+z^v}{2}\right)\left(\Pi_{j=1}^d\frac{1+z^{v^{[j]}}}{2}\right), \quad  z \in \left( \C \setminus \{0\}\right)^d, \end{equation}
determines a scheme $S_c$ which is a converging scheme, satisfying
\begin {equation}
\label{firstdiff}
\nabla_V S_c=S_b \nabla_V\ , 
\end{equation}
with $b(z)$ a diagonal matrix with diagonal elements 
\begin{equation}
\label{diag}
b_{i,i}(z)=\frac{a(z)}{2}\left(\Pi_{v\in D}\frac{1+z^v}{2}\right)\left(\Pi_{j=1,\, j\neq i}^d\frac{1+z^{v^{[j]}}}{2}\right)  ,\ \ i=1,\ldots,d,\quad  z \in \left( \C \setminus \{0\}\right)^d. 
\end{equation}
Moreover, $S_c\in \cA(1,\frac12)$.
\end{Proposition}

\begin{proof}
First we derive (\ref {firstdiff}). We denote 
$\tilde{V}^{[i]}:=V \setminus \{v^{[i]}\}$  the collection of all the directions in $V$ except for $v^{[i]}$, and by $s_D(z)$, $s_{\tilde{V}^{[i]}}(z)$
the symbols
$$ s_D(z)=\Pi_{v\in D}\frac{1+z^v}{2}\ ,\quad  s_{\tilde{V}^{[i]}}(z)=\Pi_{j=1,\, j\neq i}^d\frac{1+z^{v^{[j]}}}{2},\quad  z \in \left( \C \setminus \{0\}\right)^d.
$$
Applying the difference in direction $v^{[i]}$  to $S_c$, we get in terms of symbols
$$(z^{-v^{[i]}}-1)c(z)=(z^{-v^{[i]}}-1) \frac{1+z^{v^{[i]}}}{2}a(z)s_D(z)s_{\tilde{V}^{[i]}}(z) =(z^{-2v^{[i]}}-1)\frac{a(z)}{2}s_D(z)s_{\tilde{V}^{[i]}}(z),\quad  z \in \left( \C \setminus \{0\}\right)^d.$$
Writing the above equation in terms of operators, we get
 $\nabla_{v^{[i]}} S_c=S_{b_{i,i}}\nabla_{v^{[i]}}$ with $b_{i,i}(z)$  as in (\ref {diag}).
 Since the above holds for all $ i=1,\ldots,d $,  (\ref {firstdiff}) follows with $S_b$ a diagonal matrix with diagonal elements as in  (\ref{diag}). 
 The convergence of $S_c$ follows from (\ref{firstdiff}), (\ref{diag}) and  the observations 
\begin {enumerate}
\item 
$\|S_b\|=\max_{i=1,\ldots,d}\|S_{b_{i,i}}\|$, 
\item 
For any positive mask $e$ such that $S_e$ is a converging scheme, 
$ \|S_e\|=1\ ,$
\item 
The symbol $e_u(z)=e(z) \frac{1+z^u}{2}$ with $u\in{\mathbb N_0^d}$ and $e$ a positive mask, has positive coefficients and $\|S_{e_u}\|\le\|S_e\|$.
\end {enumerate}
Combining these observations with $ (\ref {diag})$, we get
$\|S_{b_{i,i}}\|\le\frac{1}{2}$, and therefore $\|S_b\|\le\frac{1}{2}$. This, together with $ (\ref{firstdiff})$, implies that\  $S_c\in \cA(1,\frac12)$. \end{proof}

\begin{Proposition}\label{prop:result2}
In the notation of Proposition \ref{prop:result1} and its proof, let  $ c(z)= a(z) s_D(z) s_V(z)$. Then, the first divided differences in direction $v^{[i]}$ of the data generated by $S_c$,
are mapped from one refinement level to its next refinement level  by $S_{c^*_i}$ with $c^*_i(z)=a(z)s_D(z)s_{\tV^{[i]}}(z)$.
\end{Proposition}


\begin{proof}
Applying the difference in direction $v^{[i]}$  to $S_c$, we get in terms of symbols
$$(z^{-v^{[i]}}-1)c(z)=(z^{-v^{[i]}}-1) \frac{1+z^{v^{[i]}}}{2}a(z)s_D(z)s_{\tV^{[i]}}(z) =(z^{-2v^{[i]}}-1)\frac{c^*_i(z)}{2},\quad  z \in \left( \C \setminus \{0\}\right)^d.$$
 The above equation, written in terms of operators, is
$\nabla_{v^{[i]}} S_c=S_{\frac{1}{2}c^*_i}\nabla_{v^{[i]}}.$
 Multiplying the last equation by $\frac{2^{k+1}}{\|v^{[i]}\|_2}$, we finally obtain
\begin {equation}
\frac{\nabla_{v^{[i]}}}{2^{-(k+1)}\|v^{[i]}\|_2}
S_c=S_{c^*_i}\frac{\nabla_{ v^{[i]}}}{2^{-k}\|v^{[i]}\|_2}\ , 
\end {equation} which completes the proof.
 \end{proof}

In order to state and prove our next result, we introduce the following sets of masks.  
\begin{Definition}\label{def:Nirasymbols1}
${\mathscr C_0} $ consists of masks with symbols which are products of a positive symbol of a converging stationary scheme, times $s_D(z)$ with $D$  any collection of directions in 
${\mathbb N_0^d}$,  with possible repetitions of directions. 
\end{Definition}

The next sets of masks depend on a sequence  of bases in ${\mathbb N_0^d }$, denoted by ${\mathcal V}=\{V_j\}_{j=1}^\infty$.

\begin{Definition}\label{def:Nirasymbols2}
Let  ${\mathcal V}=\{V_j\}_{j=1}^\infty$ be a sequence  of bases in ${\mathbb N_0^d }$. The sets of  masks ${\mathscr C}_j^{{\mathcal V}},\  j\ge 1$ is recursively defined from the starting set ${\mathscr C}_0^{{\mathcal V}}:={\mathscr C_0}$, as
$${\mathscr C}_j^{{\mathcal V}},   \hbox{consists of masks with symbols  that  are  products  of a symbol of a mask in}\   {\mathscr C}_{j-1}^{\mathcal V}\ \hbox{multiplied by}\ s_{V_j}(z).$$
\end{Definition}
For example, the mask $c$  with symbol in \eqref{def:c} is in  ${\mathscr C}_1^{\mathcal V} $ for any ${\mathcal V}$ with $V_1=V$.

\begin{Proposition}\label{prop:36}
The sets of masks ${\mathscr C}_j^{\mathcal V},\  j\ge 1$ satisfy
\begin {equation}\label {containment} 
 {\mathscr C}_j^{\mathcal V}\subset  {\mathscr C}_{j-1}^{\mathcal V}\ ,\ \ j\ge 1\ .
\end {equation}
Moreover, all the masks in  $ {\mathscr C}_j^{\mathcal V},\ j\ge 1$ are in $\cA(1,\frac12)$.
\end{Proposition}

\begin{proof}
It is easy to conclude from the inductive definition of the sets ${\mathscr C}_j^{\mathcal V} $ that any $ c\in {\mathscr C}_j^{\mathcal V} $  has a symbol of the form
$c(z)=a(z)s_{D_c}(z) \Pi_{i=1}^j s_{V_{i}}(z)$ where $V_1,\ldots, V_j$  are  the first  $j$ elements in the sequence ${\mathcal V}$, $a(z)$  is a symbol with positive coefficient of a converging scheme and $D_c$
is a collection of directions in ${\mathbb N_0^d}$ with possible repetitions of directions. Repeated directions are considered as different elements of $D_c$.
Including the directions  of $V_j$  in $D_c$,  we conclude that $c \in {\mathscr C}_{j-1 }^{\mathcal V} $. Thus, the masks in any ${\mathscr C}_j^{\mathcal V}$ for $j\ge 1$ are also in ${\mathscr C}_1^{\mathcal V}$, and therefore by Proposition \ref{prop:result1} these masks are in $\cA(1,\frac12)$.
\end{proof}

We are now ready to state and prove the next proposition \costi{dealing with the smoothness analysis of subdivision schemes with masks in ${\mathscr C}^{\cal V}_{m}$.}

\begin{Proposition}\label{prop:result3}
Let ${\mathcal V}=\{V_j\}_{j=1}^\infty$ be a sequence of basis in ${\mathbb N_0^d}$ and $c\in {\mathscr C}^{\cal V}_{m}$ with $m\ge 1$. Then $S_c$ generates $C^{m-1}$  limits.
\end{Proposition}

\begin{proof}The proof is by induction  on $m\ge  1$. For $m=1$ the claim  follows from Proposition \ref{prop:result1} \costi{combined with Theorem \ref{th:convergence}}.
Assume that the claim holds for some $m\ge 1$. We show that then the claim holds for $m+1$.
To prove the claim for $m+1$, we take $c\in {\mathscr  C}^{\cal V}_{m+1}$.  Then, there exists  ${\tilde c}\in {\mathscr  C}^{\cal V}_{m}$, such that $c(z)={\tilde c}(z)s_{V_{m+1}}(z)$.

Introducing the symbols  $c^{\ast}_i(z)={\tilde c}(z)s_{{\tV}^{[i]}_{m+1}}(z),\ \ i=1,\ldots,d$ with ${\tV}^{[i]}_{m+1}$   the collection of directions in $V_{m+1}$ except $v^{[i]}_{m+1}$, we conclude, in view of Proposition \ref{prop:result2} , that 
 for $i=1,\ldots,d$, the scheme  $ S_{c^*_i}$ maps the  first divided differences in the direction  $v^{[i]}_{m+1}$ of the data generated by $S_c$,
from one refinement level to its next refinement level. Since $c^*_i\in {\mathscr C}^{\cal V}_{m}$, it follows from the induction hypothesis that the 
first directional derivatives of the limits of $S_c$ in the directions of $V_{m+1}$
are in $C^{m-1}$. Since $V_{m+1}$ is a basis, all these directions are linearly indeperndent, and therefore the limits of $S_c$ are in $C^{m}$.
This completes the proof of the induction hypothesis for $m+1$.
\end {proof}

\subsection{Smoothness of non-stationary schemes via smoothing factors}
Propositions  \ref{prop:result1}, \ref{prop:result2}, \ref{prop:36} and \ref{prop:result3},  on stationary subdivision schemes enable us to prove results on non-stationary subdivision schemes with masks in the sets ${\mathscr C}_j^{{\mathcal V}},\   j\in {\mathbb N}$.

\begin{Theorem}\label{Theorem1}
Let ${\mathcal V}=\{V_j\}_{j=1}^\infty$ be a sequence of basis in ${\mathbb N_0^d}$. Let the non-stationary scheme  $(S_{c_k} :  k\in {\mathbb N}_0) $,  with  $c_k\in {\mathscr C}^{\cal V}_{j_k},\  j_k
\in {\mathbb N}$. Then the scheme is converging, and its limits are $C^{j^*-1}$,  where $j^*$ is the maximal integer satisfying $j^*\le j_k$ for all  $k\in{\mathbb N}_0$.
\end{Theorem}

\noindent  Proof: The convergence of the non-stationary scheme follows  from the convergence result, Theorem \ref{th:convergence}, in view of (\ref{containment}) and Proposition \ref{prop:result1}.
The proof of the smoothness of the limits  of $(S_{c_k}   :      k\in {\mathbb N}_0) $ is similar to the proof of Proposition \ref{prop:result3}, \costi{it is based on Proposition \ref{prop:result2} and it uses} induction on $j^*$. For $j^*=1$ the claim about smoothness follows from the convergence claim.
 Next, we assume that the claim holds for $j^*-1\ge 1$,
and prove it for $j^*$.
Since $j_k \ge  j^*$ for all $k\in {\mathbb N}_0$,  it follows from (\ref {containment}) that for all $k\in {\mathbb N}_0$ we have $c_k\in {\mathscr C}^{\cal V}_{j^*}$.
Thus,   $c_k(z)={\tilde c}_k(z)  s_{V_{j^*}}(z)$,  with $ {\tilde c}_k \in {\mathscr C}^{\cal V}_{j^*-1}$, for all $k\in {\mathbb N}_0$.
Introducing the symbols $c^*_{i,k}(z)={\tilde c}_k(z)  s_{\tV^{[i]}_{j^*}}(z)$, ,  for $k\in\N_0$ and $i=1,\ldots, d$  with ${\tV}^{[i]}_{j^*}$   the collection of directions in $V_{j^*}$ except $v^{[i]}_{j^*}$, we conclude that the corresponding masks are in 
${\mathscr C}^{\cal V}_{j^*-1}$. Moreover, by Proposition \ref{prop:result2}, the non-stationary scheme $(S_{c^*_{i,k}},\ k\in {\mathbb N})$,
 maps the first divided differences in direction $v_{j^*}^{[i]}$ of the data generated by $(S_{c_k}   :      k\in {\mathbb N}_0)  $ from one  refinement level to its next refinement level. This observation is valid for $v_{j^*}^{[1]},\ldots,v_{j^*}^{[d]}$ the vectors of  $V_{j^*}$. The convergence claim of the theorem implies that the $d$ non-stationary schemes $ (S_{c^*_{i,k}}   :      k\in {\mathbb N}_0),\  i=1,\ldots,d $ 
are converging. By the induction hypothesis these schemes   
generate limits which are in $C_{j^*-2}$
 and therefore the first directional derivatives of the limits of the scheme $ (S_{c_k}   :      k\in {\mathbb N}_0) $ in the directions of  $V_{j^*}$
are $C^{j^*-2}$. Since the $d$ directions in $V_{j^*}$  are linearly independent, the limits of  the scheme $ (S_{c_k}   :      k\in {\mathbb N}_0) $  are  $C^{j^*-1}$.

\begin{Theorem}\label{teo:2Nira}
Let ${\mathcal V}=\{V_j\}_{j=1}^\infty$ be a sequence of basis in ${\mathbb N_0^d}$. 
Let the non-stationary scheme  $(S_{c_k}   :      k\in {\mathbb N}_0) $,  with  $c_k\in {\mathscr C}^{\cal V}_{j_k},\   j_k\in {\mathbb N}$. If for some $j^*, m\in {\mathbb N},\ $
 $j_k\ge j^*$  for all  $k\ge m$, then the non-staionary scheme $(S_{c_k}   :      k\in {\mathbb N}_0)$ generates
 $C^{j^*-1}$ limits.
\end{Theorem}

\begin{proof}By Theorem \ref{Theorem1}, the non-stationary scheme   $(S_{c_{m+k}}  :      k\in {\mathbb N}_0)$   converges to  $C^{j^*-1}$  limits. Regarding \\ $S_{c_{m-1 }} S_{c_{m-2}}\ldots S_{c_1} S_{c_0} f^0 $ as
 initial data for the scheme $S_{c_m}$,  with $f^0$ an initial data tor the scheme  $(S_{c_k}   :      k\in {\mathbb N}_0) $,  we conclude the claim of the theorem.
\end{proof}

Two direct consequences of Theorem \ref{teo:2Nira} are the following corollaries. \costi{Note that Corollary \ref{coro:3} is the result needed to conclude the smoothness of our  Up-like functions.}

\begin{Corollary}\label{coro:1}
 Let the non-stationary scheme  $ (S_{c_k}   :      k\in {\mathbb N}_0) $  be as in Theorem \ref{teo:2Nira}. If the sequence $\{j_k\}_{k\in {\mathbb N}_0}$
is monotone non-decreasing and unbounded,   then the scheme  $ (S_{c_k}   :      k\in {\mathbb N}_0) $ is converging, and its limits are $C^\infty$.
\end{Corollary}

\begin{Corollary}\label{coro:3}
Let $D_0$ be a collection  of directions in ${\mathbb N}_0^d$, containing a basis $V$  of  ${\mathbb R}^d$,  and  define $a(z)=2^ds_{D_0}(z)$. If  $S_a$ is converging, then the symbols
$c_{kr+j}(z)=2^{-dk}a(z)^{k+1},\  j=0,\ldots,r-1$ for $k \in \N_0$ and some  $r\in {\mathbb N}$, define a converging non-stationary scheme $(S_{c_\ell}   :      \ell\in {\mathbb N}_0) $, generating    $C^\infty$  limits.
\end{Corollary}

\begin{proof}
Since $S_a$  is a converging scheme, and since $a(z)$ is a symbol with positive coefficients, then  
$a\in {\mathscr C}_0$. Let  $D_\ell$ denote the collection of directions corresponding to $ c_\ell(z)$.  By definition, $D_\ell$ contains at least $[\ell/r]$ repetitions of $V$, and the same number of  repetitions
of the directions in $D_0\setminus\{V\}$. Therefore, for ${\cal V}$ the constant sequence  of the basis $V$, 
we have that 
$c_\ell\in {\mathscr C}^{\cal V}_{[\ell/r]}$  for $\ell\in {\mathbb N}_0$.
 Thus, by Corollary \ref{coro:1}, the scheme $ (S_{c_\ell}   :      \ell\in {\mathbb N}_0) $  is converging to  $C^\infty$  limits.
\end{proof}

\section{The support of the Up-like functions} \label{sec:compact support}

\noindent In this section, we first state an assumption on the supports of the masks $a_k,\ k\in \N_0$ which ensures the compact
support of $\phi_0$ in \eqref{def:phim}.  Then, we compute the support of special 
non-stationary schemes based on spline masks and three directional box-spline masks, for the univariate and bivariate case, respectively. 
In our analysis we need a Lemma whose proof can be found in \cite{NIRANEW}. 

\begin{Lemma}\label{Lemma:IIINira}
Let $H$ be a compact, convex set in a Euclidean space. Then
$$\alpha H\, \Msum\, \beta H=(\alpha+\beta)\, H$$
for $\alpha, \beta$ positive reals.
\end{Lemma}

\noindent In the stationary case when $a_k=a$, $k \in \N_0$, following the subdivision process \eqref{def:sub} we see that, due to Lemma  \ref{Lemma:IIINira}, the support of the basic limit function is the set
\begin{equation*} 
\displaystyle{ {\Msum_{k=0}^\infty} 2^{-(k+1)} \supp(a)}=\supp(a)\displaystyle{ \sum_{k=0}^\infty 2^{-(k+1)}} \, =\supp(a)\,,
 \end{equation*}
  where $\Msum$ denotes the Minkowski sum of sets  (see Section \ref{BACK}).
 
\medskip  In the non-stationary case, to ensure the compact support of $\phi_0$ and to be able to compute the support of our Up-like functions,  we formulate the following assumption on the masks.

\medskip 
\noindent {\bf Assumption S:}  The sequence of masks $(a_k \, :\, k\in \N_0)$ is such that
\begin{description}
\item[$(i)$]  $\supp(a_{k})=\lambda_k \,\supp(a_0)$, $k \in \N_0,$ with $\lambda_k\in \R$, and $\lambda_0=1$;
\item[$(ii)$]  $\lambda_{k+1}\ge \lambda_k,\, k\in \N_0$;
\item[$(iii)$]  $\displaystyle \sum_{k=0}^\infty 2^{-(k+1)}\lambda_k<\infty$;

\end{description}

\begin{Remark} The conditions of {\bf Assumption S} generalize the univariate case (see e.g. \cite{CohenDyn}), where they read as
\begin{description}
\item[$(I)$] $\supp a_k=[0, N_k], \ N_k \in \N$,
\item[$(II)$] $N_k\le N_{k+1}$, $k \in \N_0$, and 
\item[$(III)$] $\displaystyle \sum_{k=0}^\infty 2^{-(k+1)} N_k < \infty$. 
\end{description}
\end{Remark}

\begin{Theorem}\label{teo:support}
Let $(a_k\,:\, k\in \N_0)$ satisfy {\bf Assumption S}. Then, the set
\begin{equation} \label{def:Sigma}
\cSigma=\displaystyle{ {\Msum_{k=0}^\infty} 2^{-(k+1)} \supp(a_k)}\,,
\end{equation} 
is a compact, convex set containing 
the support of $\phi_0=\lim_{k \rightarrow \infty} S_{a_k} S_{a_{k-1}} \ldots S_{a_0}$.
Moreover, if all masks $a_k$ have non-negative entries and if each mask define a converging stationary scheme,  then $\suppf(\phi_0)= \cSigma$.
\end{Theorem}

\begin{proof} 
A straightforward consequence of the subdivision process in \eqref{def:sub}, and the dyadic re-parametrization of the subdivision sequences by the appropriate powers of $2$, is the observation that the complement of $\cSigma$ in $\R^d$ is a set of zeros for any sequence of masks. Thus, $\suppf(\phi_0)\subseteq \cSigma$. {\bf Assumption S}, and Lemma  \ref{Lemma:IIINira} lead to
$$
\suppf(\phi_0)\subseteq \cSigma=\displaystyle{ \Msum_{k=0}^\infty 2^{-(k+1)} \supp(a_k)}
=\supp(a_0)\displaystyle{ \sum_{k=0}^\infty 2^{-(k+1)} \lambda_k}\,.
$$
 Therefore, in view of Definition \ref{def:support}, $\cSigma$ is a compact, convex set containing  $\suppf(\phi_0)$. If all the masks'  entries are non negative and if each mask defines a converging stationary scheme, then 
 there are no cancellations in the Minkoski sum in \eqref{def:Sigma} and in view of \eqref{5} all points in $\cSigma$  are non-zero, and therefore $\suppf(\phi_0)=\cSigma$.
 \end{proof}

\section{Examples}\label{sec:examples}

In this section examples  of univariate and bivariate Up-like functions are given.
\subsection{Univariate case}

\vspace{0.3cm} This subsection shows that in the univariate case there exist non-stationary subdivision schemes with $C^\infty $
basic limit functions supported on  $ [0,1+\epsilon]$, with positive $\epsilon$, arbitrarily small. \costi{We know that, to get
$C^\infty$ limits we need a growing number of the factor $\frac{1+z}{2}$  in the symbols $a_0 (z), a_1 (z), a_2
(z),\ldots $, but to decrease the support size of the basic limit function, we have to slowly increase the number of
this factor.}

Proposition \ref{prop:Nirasupport} considers \costi{the case where $r$ repetitions of the same number of smoothing factors are taken before increasing their number. For this case, we explicitly compute the support size of the corresponding basic limit function}.
Before, for completeness, we state and prove two auxiliary lemmas.

\begin{Lemma}\label{Lemma:INira}
Let $f \in \ell_\infty(\Z)$ be a sequence of positive numbers with
$supp(f)=\{0,1,2,\ldots,u\}\ , u\in \mathbb N$, and let $a(z)=\frac{(1+z)^{m+1}}{2^m}$, with $m\in \N_0$. Then
$$supp(S_a f)=\{0,1,2,\ldots, (2u+m+1)\}\  .$$
\end{Lemma}
\begin{proof}
First we observe that
\begin{equation}
\label {basic}
supp(S_a \delta)=\{0,1,2,\ldots,(m+1)\}\subset \mathbb Z\ .
\end{equation}
By the linearity of $S_a$ we obtain $S_a f=S_a\sum_{i=0}^u f(i)\delta(i-\cdot)=\sum_{i=0}^u f(i)
(S_a\delta)(i-\cdot)$.
In view of (\ref{basic}) we get that $ supp(S_a \delta(i-\cdot)=\{2i,2i+1,\ldots,2i+m+1\}$, and thus
$$ supp((S_a f))= \displaystyle \cup_{i=0}^u \{0,1,2,\ldots,2i+m+1\}=\{0,1,2,\ldots,(2u+m+1)\} ,$$
which proves the claim of the lemma.
\end{proof}


The second Lemma is more algebraic. 

\begin{Lemma}\label{Lemma:IINira}
For $q\in (0,1)$ we have $\displaystyle{\sum_{j=1}^\infty jq^j=\frac{q}{(1-q)^2}\ ,\quad \hbox{and}\quad  \ \ \sum_{j=0}^\infty (j+1)q^j=\frac{1}{(1-q)^2}}\ .$
\end{Lemma}
\begin{proof}
The first equality above follows easily from the second one. So we only prove the second equality. Since $0<q<1$,
we have
$$\sum_{j=0}^\infty (j+1)q^j=\sum_{j=0}^\infty\frac{d}{dq} q^{j+1}=\frac{d}{dq}\sum_{j=1}^\infty q^j
=\frac{d}{dq}\left(\frac{q}{1-q}\right)=\frac{1}{(1-q)^2}\ .$$
\end{proof}

\begin{Proposition} \label{prop:Nirasupport} Let $r\in \N$, and let
\begin{equation}
\label{21'}
 a_{mr}(z)=a_{mr+1}(z)=a_{mr+2}(z)=\ldots=a_{mr+r-1}(z)=\frac{(1+z)^{m+1}}{2^m},\qquad m \in \N_0 .
\end{equation}
Then, for the limit function \costi{\begin{equation}\label{blf}\phi_{0,r}=\lim_{k\rightarrow\infty}S_{a_k}S_{a_{k-1}}\ldots S_{a_0} \delta,\end{equation}} we have that 
$$
\phi_{0,r}\ \hbox{is}\  C^\infty,\quad \hbox{and}\quad \suppf(\phi_{0,r})=\left[0, 1+\frac{1}{2^r-1}\right]\,.
 $$
\end{Proposition}

\begin{proof}
Note that, for support consideration,  any subdivision sequence $f^{[k]},\ k\in\N_0$ which is defined on $\Z$, 
should be considered as defined on $2^{-k}\Z$ requiring a the reparametrization $2^{-k} f^{[k]}\subset 2^{-k}\mathbb Z$.
By Lemma \ref{Lemma:INira} and in view of \eqref {21'}
$ supp(\phi_{0,r})=[0,u_r]$ with
$$u_r=1\left(\frac{1}{2}+\frac{1}{4}+\ldots+\frac
{1}{2^{r}}\right)+2\left(\frac{1}{2^{r+1}}+\frac{1}{2^{r+2}}+\ldots+\frac{1}{2^{2r}}\right)+3\left(\frac{1}{2^{2r+1}}+\frac{1}{2^{2r+2}}+\ldots+
\frac {1}{2^{3r}}\right)+\ldots\ .$$
Thus
$$u_r=\sum_{j=0}^\infty (j+1)2^{-rj}\left(\frac{1}{2}+\frac{1}{4}+\ldots+\frac {1}{2^r}\right)=\sum_{j=0}^\infty
(j+1)2^{-rj}\sum_{\ell=1}^r 2^{-\ell} .$$
In view of Lemma \ref{Lemma:IINira} we finally get
$$u_r=\frac{2^r-1}{2^r}\frac{1}{(1-2^{-r})^2}=\frac{2^r}{2^r-1}=1+\frac{1}{2^r-1}\ .$$
Since the masks $a_{m}\in {\mathscr C}_{[m/r]}^{\cal V}$ with ${\cal V}=\{1,1, 1,...\}$, we conclude based on Corollary \ref{coro:3}
that $\phi_{0,r}$, defined in \eqref{blf}, is $C^{\infty}$.
This completes the proof of the proposition. 
\end{proof}

Based on the previous results, we can construct the $C^\infty$ compactly supported Up-function supported in $[0,2]$, 
as well as other $C^\infty$  functions with a smaller support than the support of the Up-function that are a new type of Up-like functions.

\begin{Example}[Up-function]\label{example:UP}
We consider the case $r=1$ that is the symbols
 $$a_{m}(z)=\frac{(1+z)^{m+1}}{2^m},\ m\ge 0.
 $$ 
Based on Proposition \ref{prop:Nirasupport}, the  basic limit function  $\phi_{0,1}=\lim_{k\rightarrow\infty}S_{a_k}S_{a_{k-1}}\ldots S_{a_0} \delta$ is $C^\infty$ and has support
$\left[0, 2\right]$.
\end{Example}

\begin{Example}[The Up-like function $\phi_{0,2}$]\label{example:NIP}
We start by considering the case $r=2$ where 
 $$a_{2m}(z)=a_{2m+1}(z)=\frac{(1+z)^{m+1}}{2^m},\ m\ge 0.
 $$ 
Based on Proposition \ref{prop:Nirasupport}, the basic limit function  $\phi_{0,2}=\lim_{k\rightarrow\infty}S_{a_k}S_{a_{k-1}}\ldots S_{a_0} \delta$ is $C^{\infty}$ with support
$\left[0, \frac{4}{3}\right]$.
\end{Example}

\subsection{Bivariate case}
In this subsection we discuss bivariate examples which  are a natural generalization of the previous univariate example.
For  a given 
$ r\in \mathbb {N}$ we consider the  non-stationary subdivision scheme with  masks of the form
\begin{equation}
\label{masks00}
 a_{mr}=a_{mr+1}=\ldots=a_{(m+1)r-1}=s_{m+1}(e^{[1]},  e^{[2]}, e^{[3]}),\  \  m\in\mathbb{N}_0\ , 
\end{equation}
where $e^{[1]}=(1,0),\ e^{[2]}=(0,1),\ e^{[3]}=(1,1)$ and $s_m(e^{[1]} ,e^{[2]},e^{[3]})$ denotes the mask  of the three-directional box-spline
corresponding to the directions $e^{[i]}\ i=1,2,3$, each repeated $m$ times. In case $m=1$  we omit the $m$, and denote the corresponding mask by  $s(e^{[1]}, e^{[2]}, e^{[3]})$. 
For this non-stationary scheme a support result can be proven based on Theorem \ref{teo:support} and two lemmas.
The first  is Lemma \ref{Lemma:IIINira}
and the second is (see \cite[Chaper 1]{Box}).
\begin{Lemma}\label{Lemma:IVNira}
For $m\in \mathbb{N}$, $\supp(s_ m( e^{[1]}, e^{[2]}, e^{[3]}))$ is a convex set satisfying
 $$\supp(s_m( e^{[1]}, e^{[2]}, e^{[3]}))=m\cdot \supp(s(e^{[1]}, e^{[2]}, e^{[3]}))\ .$$
\end{Lemma}

\begin{Proposition} \label{prop:NirasupportBIv}
Let $r\in \N$. The non-stationary subdivision scheme with the masks in \eqref{masks00} generates a $C^\infty$ bivariate  basic limit function $\phi_{0,r}=\lim_{k\rightarrow\infty}S_{a_k}S_{a_{k-1}}\ldots S_{a_0} \delta$ with support
  $$ \suppf(\phi_{0,r})=\left(1+\frac{1}{2^r-1}\right)\supp(s(e^{[1]},  e^{[2]},  e^{[3]})) \ .$$
\end{Proposition}

\begin{proof}  By Theorem \ref{teo:support},
$$ supp(\phi_{0,r})=\Msum _{j=0}^\infty 2^{-(j+1)} \supp(a_j)=\Msum_{m=0}^\infty\Msum_{j=0}^{r-1} 2^{-(mr+j+1)} 
\supp(s_{m+1}(e^{[1]},  e^{[2]},  e^{[3]}))\ .$$ 
Thus, in view of Lemma \ref{Lemma:IIINira} and Lemma \ref{Lemma:IVNira},
$$ supp(\phi_{0,r})=\supp(s(e^{[1]},  e^{[2]},  e^{[3]}))\sum_{m=0}^\infty (m+1) 2^{-rm}\sum_{j=0}^{r-1} 2^{-(j+1)}\ .$$
Since $\sum_{j=0}^{r-1}2^{-(j+1)}=1-2^{-r}$,  and since by Lemma  \ref{Lemma:IINira}, $\sum_{m=0}^\infty (m+1) 2^{-rm}=\left(\frac{1}{1-2^{-r}}\right)^2$,
we finally get
$$ \suppf(\phi_{0,r})=\left(1+\frac{1}{2^r-1}\right)\supp(s(e^{[1]},  e^{[2]},  e^{[3]})).$$
That $\phi_{0,r}$ is $C^\infty$ follows from the observation that, by \eqref{masks00}, $a_m\in {\mathscr C}^{\cal V}_{[m/r]}$, $m\in\N$  with ${\cal V}=\{V, V, V, \ldots\}$, $V=\{(1,0),(0,1)\}$, and from Corollary \ref{coro:3}.
This completes the proof of the proposition.
\end{proof}

\begin{Remark}It is interesting to note the rapid decay of  $\suppf(\phi_{0,r})$  with increasing $r$: for $r=1$ we have $\suppf(\phi_{0,1})=2 \supp(s(e^{[1]},  e^{[2]},  e^{[3]}))$,  for $r=2$
$\suppf(\phi_{0,2})=\frac{4}{3} \supp(s(e^{[1]},  e^{[2]},  e^{[3]}))$, and   for $r=5$,  $\suppf(\phi_{0,5})=\frac{32}{31} \supp(s(e^{[1]},  e^{[2]},  e^{[3]})).$
\end{Remark}

\begin{Example}\label{phi01}
 
\medskip We start with the bivariate analogous of the Up-function where the role of schemes generating polynomial splines is replaced by three-directional Box-splines. 
It corresponds to the choice of $r=1$ that is to the  non-stationary subdivision scheme with  symbols of the form
\begin{equation}\label{masks}
 a_{m}(z_1,z_2)=\frac{\left((1+z_1)(1+z_2)(1+z_1z_2)\right)^{m+1}}{2^{3m+1}},\  \  m\in\mathbb{N}_0\ .
\end{equation}
For this non-stationary scheme, we conclude from Proposition \ref{prop:NirasupportBIv} that $\phi_{0,1}$ is $C^\infty$ and that 
$\suppf(\phi_{0,1})=2\, supp(B_{111})$, with $B_{111}$  the Box-spline corresponding to one repetition of $e^{[i]},\ i=1,2,3.$ Thus, the support of $\phi_{0,1}$ is the hexagonal-shaped domain with vertices $(0,0),\ (2,0),\ (4,2),\ (4,4),\ (2,4),\ (0,2)$, taken counterclockwise.
A graph of $\phi_{0,1}$ obtained by running 4 steps of the corresponding non-stationary subdivision scheme is depicted  in Figure  1
together with $\suppf(\phi_{0,1})$.

\begin{figure}[ht]
 \label{fig:figure1} 
 \center
\includegraphics[scale=0.3]{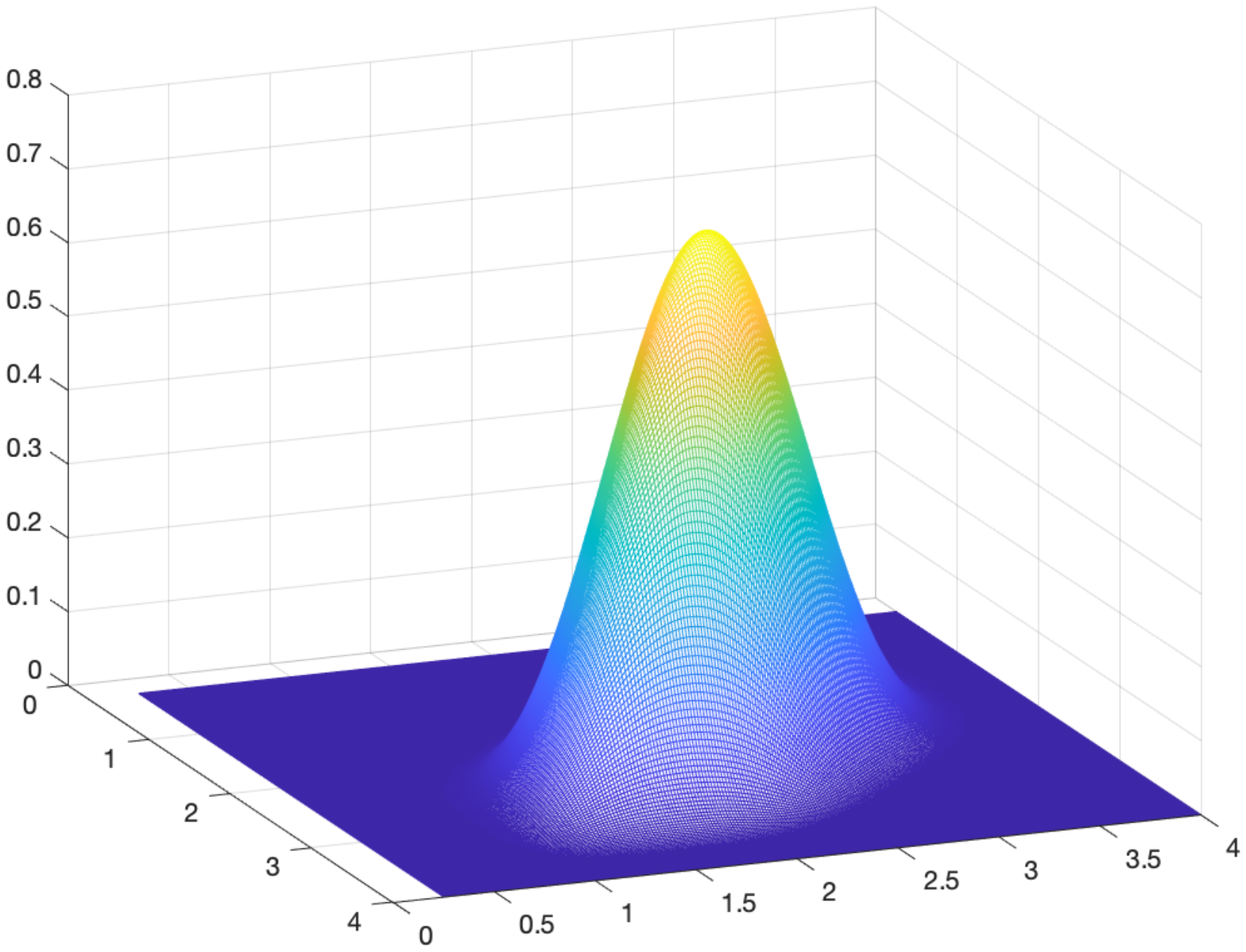}%
\includegraphics[scale=0.3]{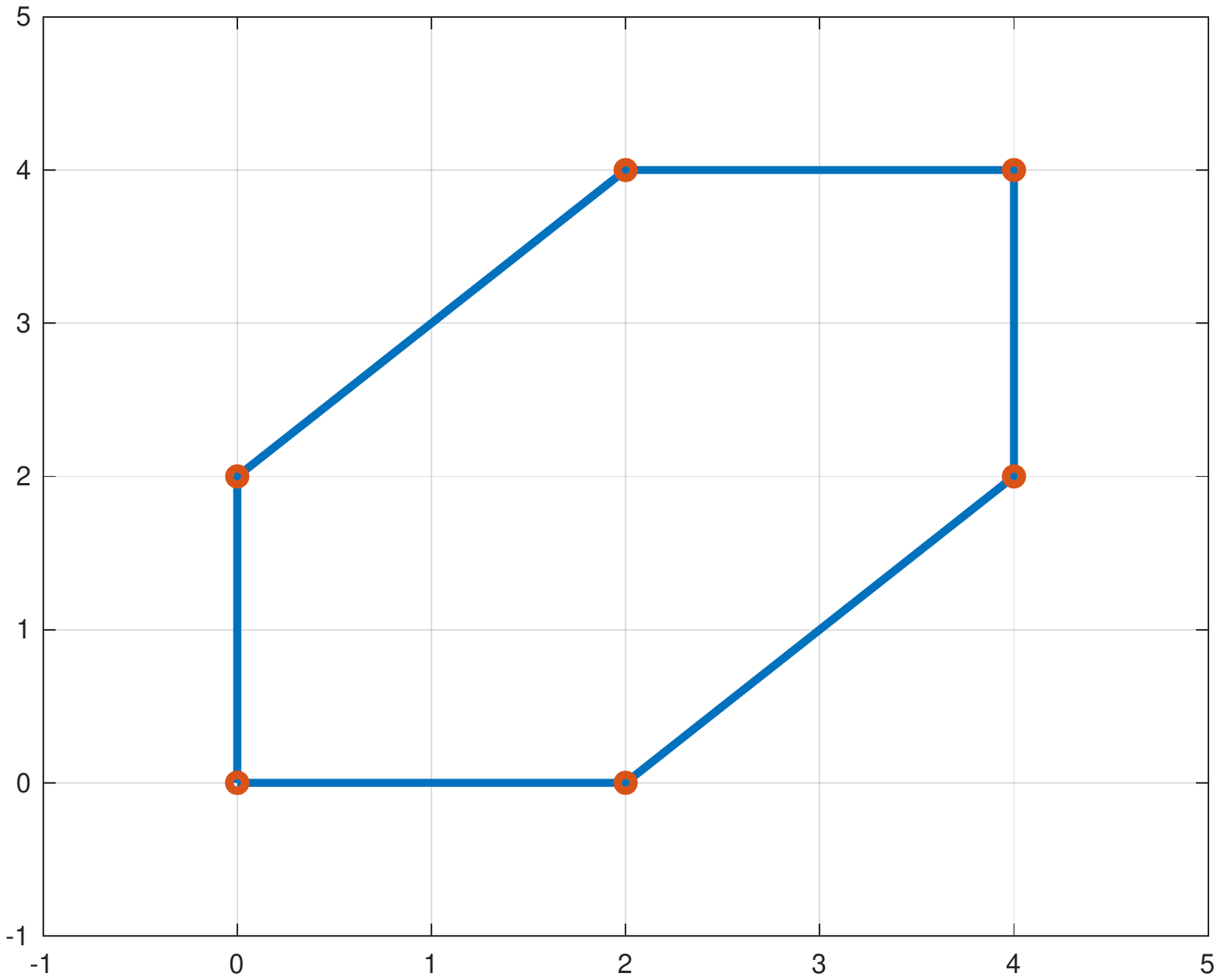}
\vskip -2cm
\caption{Graph of the function $\phi_{0,1}$ (left) and its support (right).}
\end{figure}

\end{Example}

\begin{Example}
We continue with the bivariate analogous of the function discussed in Example \ref{example:NIP}. Here again, the role of B-splines is replaced by three-directional Box-splines and  $r=2$. In other words, we consider  the  non-stationary subdivision scheme with symbols of the form
\begin{equation}\label{masks}
 a_{2m}(z_1,z_2)=a_{2m+1}(z_1,z_2)=\frac{\left((1+z_1)(1+z_2)(1+z_1z_2)\right)^{m+1}}{2^{3m+1}},\  \  m\in\mathbb{N}_0\ .
\end{equation}
For this non-stationary scheme, using  Proposition \ref{prop:NirasupportBIv} we see that $\phi_{0,2}$ is $C^\infty$ and that 
$ supp(\phi_{0,2})=\frac{4}{3}\, supp(B_{111})$, which is 
 the hexagonal-shaped domain with vertices $(0,0),\ (\frac43,0),\ (\frac83,\frac43),\ (\frac83,\frac83),\ (\frac43,\frac83),\ (0,\frac43)$, taken counterclockwise.
We give in Figure  2
the graph of $\phi_{0,2}$ obtained by running $8$ steps of the corresponding non-stationary subdivision scheme. We also give  the graph of the support of $\phi_{0,2}$.

\begin{Remark}
Note that in case $r=2$ four steps of the non -stationary scheme are not sufficient to show the smoothness of the limit function.
\end{Remark}

\begin{figure}[ht]
 \label{fig:figure2} 
 \center
\includegraphics[scale=0.3]{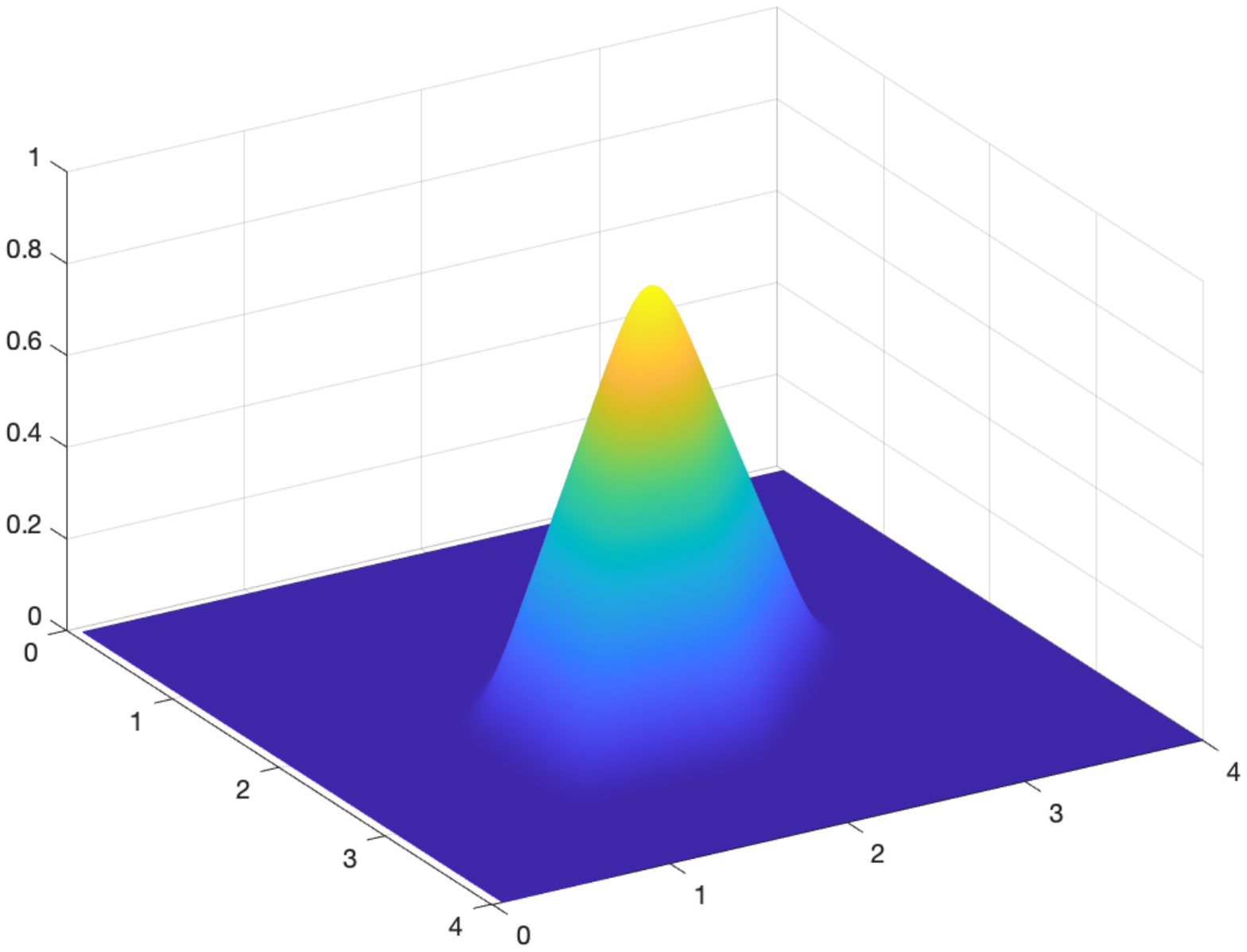}%
\includegraphics[scale=0.3]{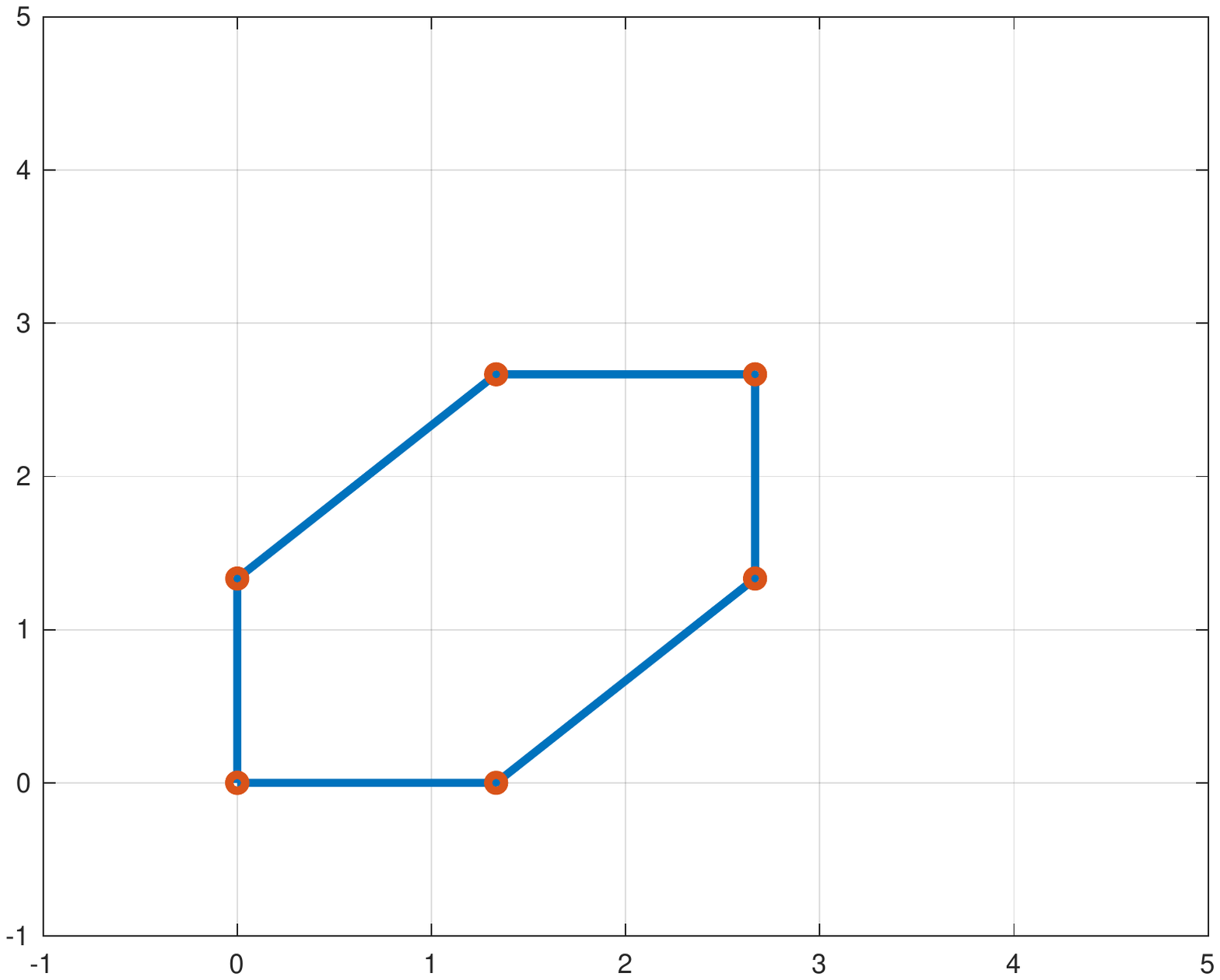}
\vskip -2cm
\caption{Graph of the function $\phi_{0,2}$ (left) and its support (right).}
\end{figure}

\end{Example}

 \bigskip {\bf Acknowledgement}\\
The second author is a member of INdAM-GNCS partially supporting this work, of Rete ITaliana di Approssimazione, and of UMI-T.A.A. group. 


\begin{thebibliography}{99}

\bibitem{Box} C. de Boor, K. Höllig, S, Riemenschneider, Box Splines,  Applied Mathematical Sciences (AMS, volume 98) Springer 1993.

\bibitem{CHM1} C. Cabrelli, C. Heil, U. Molter, Accuracy of lattice translates of several multidimensional 
refinable functions, J. Approx. Theory 95 (1998), 5-52.

\bibitem{CDM91} A.S. Cavaretta, W. Dahmen, C.~A.~Micchelli, Stationary
Subdivision, Mem. Amer. Math. Soc., 453 (1991) i-vi; 1-185.

\bibitem{CharinaContiSauer} M. Charina, C. Conti, T. Sauer, Regularity of multivariate vector subdivision schemes,
Numer. Algorithms  39 (2005), 97-113.

  \bibitem{CohenDyn} A. Cohen, N. Dyn, Nonstationary subdivision schemes
and multiresolution analysis, SIAM J. Math. Anal. 27 (1996), 1745-1769. 

\bibitem{ContiDyn2019} C. Conti,  N. Dyn, Non-stationary subdivision schemes: state of the art and perspectives, Springer Proceedings in Mathematics and Statistics, Volume 336, 39 - 71 (2021) 

\bibitem{Dau} I. Daubechies, Ten Lecture on wavelets, CBMS-NSF Regional Conference Series in Applied Mathematics, (1992) 

\bibitem{DynLevinDerfel} G. Derfel,  N. Dyn, D. Levin, Generalized refinement equations and subdivision processes, 
J. Approx. Theory 80 (1995),  272-297.	

\bibitem{DongShen} B. Dong, Z. Shen, Pseudosplines, wavelets and framelets, Appl. Comput. Harmon. Anal.
22 (2007), 78–104.

\bibitem{Dubuc} G. Deslauriers, S. Dubuc, Symmetric iterative interpolation processes, Constr. Approx. 5 (1989) 49–68
	
\bibitem{Dyn} N. Dyn,  Subdivision schemes in CAGD, in Advances in Numerical Analysis Vol. II: 
Wavelets, Subdivision Algorithms and Radial Basis Functions, W. A. Light, (ed), Oxford University Press, Oxford, 1992, 36-104.

\bibitem{DynLevin_asymp} N. Dyn, D. Levin, Analysis of asymptotically equivalent binary subdivision schemes,
Math. Anal. Appl. 193 (1995),  594-621.


\bibitem{DynLevin_acta} N. Dyn, D. Levin, Subdivision schemes in geometric modelling. Acta Numerica, 11,  (2002), 73-144.

\bibitem{HanShen} B. Han, Z. Shen, Compactly supported symmetric $C^\infty$ wavelets with
spectral approximation order, SIAM J. Math. Anal. 40 (2008), 905-938.
  
 \bibitem{Levin} D. Levin, Using Laurent polynomial representation for the analysis of non-uniform binary subdivision schemes, Adv.
Comp. Math. 11 (1999), 41-54.	


\bibitem{NIRANEW}  R. T. Rockafellar, Convex analysis. Vol. 18. Princeton university press, 1970.
	
\bibitem{Rvachev} V. A. Rvachev, Compactly supported solutions of functional-differential equations and their applications,  
Russian Mathematical Surveys 45.1 (1990): 87

\end{thebibliography}
\end{document}